\documentclass[a4paper,12pt, twoside]{amsart}

\DeclareMathSizes{11}{11}{8}{7}

\usepackage{amsmath}
\usepackage{amssymb}
\usepackage{amsthm}
\usepackage{bbm}
\usepackage{mathrsfs}
\usepackage{enumerate}
\usepackage[utf8]{inputenc}
\usepackage{yfonts}
\usepackage{nicefrac}

\usepackage[all]{xy}
\usepackage{graphicx}
\usepackage{epsfig}
\usepackage{pinlabel}

\usepackage{fancyhdr}
\newtheorem{thm}{Theorem}[section]
\newtheorem{prop}[thm]{Proposition}
\newtheorem{lemma}[thm]{Lemma}
\newtheorem*{thm*}{Theorem}
\newtheorem*{lemma*}{Lemma}
\newtheorem*{prop*}{Proposition}
\newtheorem{cor}[thm]{Corollary}
\newtheorem*{cor*}{Corollary}
\newtheorem*{fact*}{Fact}

\newtheorem*{conjecture*}{Conjecture}

\theoremstyle{definition}
\newtheorem{definition}[thm]{Definition}
\newtheorem*{def*}{Definition}
\newtheorem{rmk}[thm]{Remark}
\newtheorem*{rmk*}{Remark}

\newtheorem{bem*}{Bemerkung}

\newtheorem*{ex*}{Example}

\newcommand{\N}{\mathbbm{N}}
\newcommand{\Z}{\mathbbm{Z}}
\newcommand{\Q}{\mathbbm{Q}}
\newcommand{\R}{\mathbbm{R}}

\newcommand{\hyp}{\mathbbm{H}}

\newcommand{\scrF}{\mathscr{F}}

\newcommand{\calE}{\mathcal{E}}

\newcommand{\calT}{\mathcal{T}}
\newcommand{\calU}{\mathcal{U}}

\newcommand{\frakT}{\mathfrak{T}}

\newcommand{\al}{\alpha}
\newcommand{\bb}{\beta}

\newcommand{\ee}{\varepsilon}
\newcommand{\ff}{\varphi}
\newcommand{\ga}{\gamma}

\newcommand{\dist}{\operatorname{dist}}

\newcommand{\im}{\operatorname{im}}

\newcommand{\sub}[1]{_{#1}}
\newcommand{\ind}[2]{_{#1}^{#2}}

\newcommand{\cart}{\hspace*{-1mm}\times\hspace*{-1mm}}
\newcommand{\dirsum}{\hspace*{-1mm}\oplus\hspace*{-1mm}}
\newcommand{\minus}{\hspace*{-0.8mm}\setminus\hspace*{-0.8mm}}

\newcommand{\lra}{\longrightarrow}

\newcommand{\betavol}{\operatorname{vol}_{\beta}} %beta-Volumen
\newcommand{\mvolbeta}{\operatorname{MVol}_{\beta}}

\title{Coarse homology of leaves of foliations}
\author{Robert Schmidt}

\begin{document}

\begin{abstract}
We investigate the coarse homology of leaves in foliations of compact manifolds.  This is motivated by the observation that the non-leaves constructed by Schweitzer and by Zeghib all have non-finitely generated coarse homology.  This led us to ask whether the coarse homology of leaves in a compact manifold always has to be finitely generated.  We show that this is not the case by proving that there exist many leaves with non-finitely generated coarse homology.  Moreover, we improve Schweitzer's non-leaf construction and produce non-leaves with trivial coarse homology.
\end{abstract}

\maketitle

\tableofcontents

\section{Introduction}

Given a Riemannian manifold $L$, it is in general very hard to determine whether there exists a foliation of a compact manifold $M$ such that $L$ is quasi-isometric to one of the leaves equipped with the induced metric from $M$.  However, we can rule out certain Riemannian manifolds.  $6$-dimensional examples  of non-leaves were given by Attie and Hurder \cite{Attie-Hurder: Manifolds which cannot be leaves of foliations}, and Zeghib \cite{Zeghib: An expample of a 2-dimensional no leaf} modified their construction to produce $2$-dimensional non-leaves. Schweitzer showed in 1994 \cite{Schweitzer: Surfaces not quasi-isometric to leaves in codimension one foliations} that every open surface carries a metric of bounded geometry that cannot be bi-Lipschitz equivalent to a leaf in a foliation of a compact $3$-manifold, and in 2009 he was able to extend his results from dimension $2$ to any dimension \cite{Schweitzer: Riemannian manifolds not quasi isometric to leaves in codimension one foliations}.  He showed that 
every non-compact manifold carries a metric of bounded geometry such that the resulting Riemannian manifold cannot be diffeomorphically quasi-isometric to a leaf in a codimension one foliation of a compact manifold.  

The above-mentioned non-leaves are constructed through manifolds that violate certain conditions met by leaves in foliations of compact manifolds, and which are preserved under quasi-isometries.  Schweitzer's \textit{bounded homology property} of a Riemannian manifold $(M,g)$ is a condition on certain types of volumes of nullhomologous hypersurfaces in $M$. We give a brief description of the bounded homology property in Section \ref{section: Schweitzers bounded homology property}, for more details we refer the reader to \cite{Schweitzer: Riemannian manifolds not quasi isometric to leaves in codimension one foliations}. Attie-Hurder and Zeghib use the so called \textit{geometric entropy} of a metric space, a measure of the complexity of coverings of the space. The bounded homology property is specifically tailored to foliations and might at first be difficult to grasp. It is thus an interesting question whether the bounded homology property can be expressed through established quasi-isometry invariants.

Schweitzer's and Zeghib's non-leaf constructions deform a metric on a Riemannian manifold (in Schweitzer's case an arbitrary metric on an arbitrary manifold, in Zeghib's case the hyperbolic metric on $\hyp^2$) by inserting balloons of radius tending towards infinity.  Computing the coarse homology of the non-leaves, one notices that it is never finitely generated (Proposition \ref{prop: g_SJ has infinitely generated coarse homology}). This suggests that the bounded homology property and possibly other conditions on leaves might be expressible through the number of generators of the coarse homology. The only known non-leaf with finitely generated coarse homology is the $6$-dimensional example given by Attie and Hurder but no non-leaves with trivial coarse homology were known before. What we want to find out is hence whether the coarse homology of a leaf in a compact manifold must always be finitely generated, and conversely, whether there exist non-leaves with trivial coarse homology.  As we will now describe,
 the first question is to be answered in the negative while we can give a positive answer to the latter.

Due to a connection between coarse homology and ends of manifolds, we can show that there exist leaves in foliations of compact manifolds that have non-finitely generated coarse homology.

\begin{thm}\label{thm: leaves with non-finitely generated HX1}
In every dimension greater than or equal to $2$ there exist Riemannian manifolds $L$ with the degree $1$ coarse homology $HX_1(L)$ containing an Abelian subgroup of infinite rank, such that $L$ can be realized as a leaf in a foliation of a compact manifold of any codimension.
\end{thm}

All known non-leaf constructions on an arbitrary manifold produced metrics with non-finitely generated coarse homology. Building on the work of Schweitzer we are able to give a non-leaf construction starting with any non-compact Riemannian manifold that does not affect the coarse homology. 

\begin{thm}\label{thm: non-leaves with unchanged coarse homology}
 On every non-compact Riemannian manifold of bounded geometry $(M ,g)$, $\dim M\geq 3$, there exists a deformation of $g$ to a bounded geometry metric $g'$ such that $(M,g')$ cannot be diffeomorphically quasi-isometric to a leaf of a codimension one $C^{2,0}$-foliation of a compact manifold.  This deformation can be performed in such a way that the coarse homology and the growth type of $(M,g)$ remain unchanged.
\end{thm}

Applying the theorem to any Riemannian manifold with trivial coarse homology, e.g. a one-ended cylinder $\partial D^n\times \R\sub{\geq 0}\cup D^n\times \{0\}$ we find the following corollary.

\begin{cor}\label{cor: non-leaves with trivial coarse homology}
In every dimension $n\geq 3$, there exist Riemannian manifolds with trivial coarse homology which are not diffeomorphically quasi-isometric to a leaf of a codimension one $C^{2,0}$-foliation of a compact manifold.
\end{cor}

Theorem \ref{thm: non-leaves with unchanged coarse homology} is optimal in the sense that we cannot expect to find a non-leaf metric of bounded geometry with trivial coarse homology on every open manifold since the degree $1$ coarse homology might be non-trivial for topological reasons as is shown in Proposition \ref{prop: ends create degree 1 coarse homology}.
\\

\textbf{Convention:}  Throughout this article, all foliated manifolds are compact and the foliations are of codimension one.  In particular, the statement \textit{``$(L,g)$ is not quasi-isometric to a leaf''} means \textit{''There doesn't exist a compact manifold $M$ and a codimension $1$ foliation $\scrF$ of $M$ such that $(L,g)$ is quasi-isometric to a leaf of $\scrF$}''.  Moreover, we take all Riemannian metrics to be of bounded geometry and all manifolds to be connected and without boundary.
\\

\textbf{Acknowledgements:} This article presents results from my dissertation \cite{Schmidt: Coarse topology of leaves of foliations}, which was supervised by D. Kotschick. I would like to thank him for his advice and continuous support.

\section{Schweitzer's bounded homology property}\label{section: Schweitzers bounded homology property}

This section gives a brief overview of the results of Schweitzer's article \cite{Schweitzer: Riemannian manifolds not quasi isometric to leaves in codimension one foliations} which we will use in Section \ref{section: non-leaves with trivial coarse homology} to give a construction of non-leaves with trivial coarse homology. The following definitions and results are all from that article.  We will also briefly recall Schweitzer's construction of the non-leaf metric $g_S$ at the end of this section.

\begin{thm}[Theorem 2.8, \cite{Schweitzer: Riemannian manifolds not quasi isometric to leaves in codimension one foliations}]\label{thm:  Schweitzer exists nonleaf metric}
 Let $(L,g)$ be an open Riemannian manifold of bounded geometry. Then there exists a bounded geometry metric $g_S$ on $L$ such that $(L,g_S)$ cannot be diffeomorphically quasi-isometric to a leaf in a codimension $C^{2,0}$-foliation of codimension one of a compact manifold.  Moreover, $g_S$ can be chosen such that $g$ and $g_S$ have the same growth type.
\end{thm}

Theorem \ref{thm:  Schweitzer exists nonleaf metric} follows from the fact that every manifold which is diffeomorphically quasi-isometric to a leaf has to satisfy the bounded homology property. The definition of the bounded homology property requires a little preparation:

\begin{definition}[$\beta$-volume]
 Let $S$ be a subset of a metric space and $\beta>0$.  The \textit{$\beta$-volume} $\betavol(S)$ of $S$ is defined as the minimal number of balls of radius $\bb$ needed to cover $S$.  We have $\betavol(S)\in \N\cup\{\infty\}$.
\end{definition}

\begin{definition}[Morse-$\beta$-volume]
Let $(C,g)$ be a compact Riemannian manifold with boundary and $f: C\rightarrow [0,\infty)$ a Morse function satisfying $f|_{\partial C}\equiv 0$.  For $\beta>0$, the \textit{Morse-$\beta$-volume of $C$ with respect to $f$} is defined to be the smallest natural number $\mvolbeta(C,f)$ such that the $\bb$-volume of every level set of $f$ is bounded by $\mvolbeta(C,f)$, that is $\betavol(f^{-1}(t))\leq \mvolbeta(C,f)$ for all $t\geq 0$.

The \textit{Morse-$\beta$-volume} of $C$ is then defined to be the minimum of $\mvolbeta(C,f)$ taken over all Morse functions vanishing on $\partial C$.  In formulae, the Morse-$\bb$-volume is defined as
\[
\mvolbeta(C)=\min\sub{\genfrac{}{}{0pt}{}{f|\sub{\partial C\equiv 0}}{f\geq 0 \text{ Morse}}} \max\sub{t\geq 0}\betavol\left(f^{-1}(t)\right).
\]
\end{definition}

\begin{definition}[bounded homology property]\label{def: bounded homology property}
 A Riemannian manifold $M$ has the \textit{bounded homology property} if for all $k>0$ and all sufficiently large $\beta>0$, there exists a constant $K(\beta, k)$ such that the Morse-$\bb$-volume $\mvolbeta(C)$ of all compact codimension $0$ submanifolds $C$ with smooth boundary is bounded by $K(\beta, k)$, provided they satisfy the following conditions:
\begin{enumerate}[i)]
 \item $\betavol(\partial C)\leq k$,
 \item $C$ and $\partial C$ are connected and simply connected,
 \item $\partial C$ has a tubular neighbourhood $V$ that contains 
\[
 B_{\bb}(\partial C)=\{x\in M\mid \dist(x, \partial C)<\bb\}.
\]
\end{enumerate}
\end{definition}

\begin{thm}[Theorem 2.6, \cite{Schweitzer: Riemannian manifolds not quasi isometric to leaves in codimension one foliations}]\label{thm: Schweitzer leaves satisfy bounded homology}
 Every $n$-manifold, $n\geq 3$, that is diffeomorphically quasi-isometric to a leaf of a codimension one $C^{2,0}$-foliation of a compact manifold satisfies the bounded homology property.
\end{thm}

\begin{figure}
 \centering
 \includegraphics[height=9.5cm]{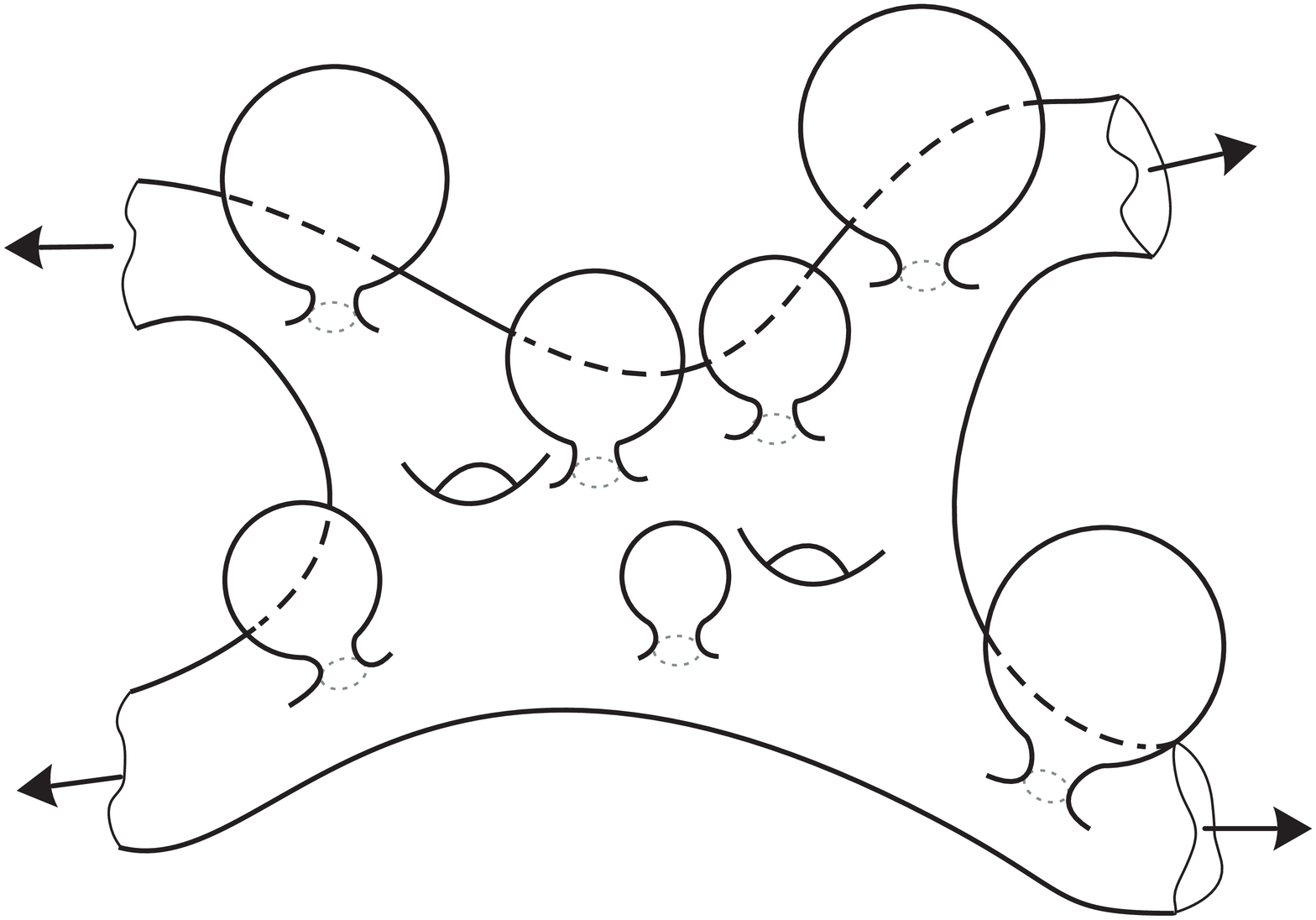}
\caption{$L$ with Schweitzer's non-leaf metric.\protect\footnotemark}
\label{fig: balloons attached}
\end{figure}
\footnotetext{The picture is taken from Schweitzer's article \cite{Schweitzer: Riemannian manifolds not quasi isometric to leaves in codimension one foliations}. I kindly thank him for giving his permission to reproduce it here.}

The non-leaves by Schweitzer are constructed by deforming any given metric to a metric not satisfying the bounded homology property. The deformation is done by inserting balloons: Since $(L,g)$ has bounded geometry, the injectivity radius is bounded from below by some $d>0$.  Hence every metric $d$-ball is topologically a ball.  Now choose a sequence of real numbers $d_i$ such that $d_i+2d<d\sub{i+1}$ and let $x_i$ be points in $L$ such that $d(x_0,x_i)=d_i$.  Choose moreover a sequence of real numbers $r_i$ converging to infinity.  And let $S^n(r_i)\minus B\sub{\nicefrac{d}{2}}(S)$ be the spheres of radius $r_i$, where a ball of radius $\nicefrac{d}{2}$ has been removed around the south pole.  Now replace every ball $B(x_i,\nicefrac{d}{2})$ in $L$ by the ``balloon'' $S^n(r_i)\minus B\sub{\nicefrac{d}{2}}(S)$ and on the annulus $B_d(x_i)\minus B\sub{\nicefrac{d}{2}}(x_i)$ interpolated smoothly between the round metric of $S^n(r_i)$ and the original metric $g$ on $L$.  This defines the balloon metric $g\sub{S}
$ on $L$ (
see Figure \ref{fig: balloons attached}).  Note that topologically we have  just replaced a ball by a ball and hence the new manifold is diffeomorphic to $L$.

It is now not hard to see that $(L,g_S)$ does not satisfy the bounded homology property, for the balloons $S^n(r_i)\minus B\sub{\nicefrac{d}{2}}(S)$ form a sequence of submanifolds $C_i$ with $\betavol(\partial C_i)=\betavol(\partial B\sub{\nicefrac{d}{2}}(S)$ but $\mvolbeta(C_i)$ tends to infinity as $i$ does.

\section{Computations of coarse homology}

In this section we briefly recall the most important definitions and facts about coarse homology and show that the coarse homology of the non-leaves constructed by Schweitzer is non-finitely generated (Proposition \ref{prop: g_SJ has infinitely generated coarse homology}). For more details about coarse homology, we refer the reader to \cite{Higson-Roe: On the coarse Baum-Connes conjecture} and \cite{Roe: Lectures on Coarse Geometry}. Whenever possible, we will omit to mention the coefficient ring of the homology theories.  All the results we prove hold for  arbitrary coefficients, though.

Recall that the \textit{locally finite homology} $H_k^{lf}(Z)$ of a locally compact topological space $Z$ is the homology of the chain complex based locally finite chains, i.e.\  possibly infinite formal sums of singular simplices $\sum\sub{\sigma}\al\sub{\sigma}\sigma, \sigma\colon \Delta^k\lra Z$ such that every compact subset of $Z$ intersects at most finitely many $\sigma$ with $\al\sub{\sigma}\neq 0$. The ordinary boundary map on singular simplices induces a well-defined boundary map on locally finite chains.

A \textit{coarsening sequence} of a proper metric space $X$ is a sequence of locally finite open coverings $\calU_1, \calU_2,\ldots$ such that the diameter of the sets in $\mathcal{U}_i$ is bounded from above by a constant $R_i$ and that the Lebesgue number of $\mathcal{U}\sub{i+1}$ is at least $R_i$.  Moreover, $R_i$ tends to infinity as $i$ does.  We denote by $\vert \mathcal{U}_i\vert$ the nerve of the covering $\mathcal{U}_i$, that is the simplicial complex with vertices $(U)$ given by the sets $U\in\calU_i$ and $k$-simplices $(U_0, \ldots, U_k)$ spanned by $U_0, \ldots, U_k\in \calU_i$ with $U_0\cap \ldots \cap U_k\neq \emptyset$.  Note that each $U_j\in \mathcal{U}_i$ lies in some $V_l\in \mathcal{U}\sub{i+1}$ and the choice of such an assignment $U_j\mapsto V_l$ induces a proper map $\vert \mathcal{U}_i\vert \rightarrow \vert \mathcal{U}\sub{i+1}\vert$.  In what follows, we will fix such an assignment and the induced maps will be called the \textit{coarsening maps}.  We will use term coarsening 
sequence both for the sequence of locally finite coverings $\calU_1, \calU_2, \ldots$ and for the sequence of their geometric realizations together with the coarsening maps $\vert \calU_1\vert \rightarrow \vert \calU_2\vert \rightarrow \ldots$  Since the coarsening maps are proper, the induced homomorphisms on locally finite homology give rise to a direct system
\[
 H\ind{*}{lf}(\vert \calU_1\vert)\lra H\ind{*}{lf}(\vert \calU_2\vert)\lra H\ind{*}{lf}(\vert \calU_3\vert) \lra\ldots 
\]

A very practical coarsening sequence is given as follows:  Let $Y$ be a $1$-dense subset of $X$, that is for any $x\in X$, there exists a $y\in Y$ such that $d_X(x,y)< 1$.  Assume further that $Y$ has no accumulation points.  Then for any radius $i\geq 1$, the collection of open balls of radius $i$, $\mathcal{B}_i(Y):=\{B_i(y)\}\sub{y\in Y}$ forms a locally finite open covering of $X$ with Lebesgue number at least $i-1$.  By letting the radii range over all natural numbers, we obtain a coarsening sequence $\left\{ \mathcal{B}_i(Y) \right\}\sub{i\in \N}$.  We set $\vert \mathcal{B}_i(Y)\vert=R_i(X;Y)$, but we will henceforth simply write $R_i(X)$ whenever the choice of a $1$-dense subset $Y$ as above is implicit.  We have natural inclusions $B_i(y)\hookrightarrow B\sub{i+1}(y)$ and the induced coarsening maps hence are just the inclusion of a subcomplex $R_i(X)\hookrightarrow R\sub{i+1}(X)$.  (The notation $R_i(X)$ is slightly abusive, since it usually denotes the $i$th Rips complex of $X$ in which the 
simplices are spanned by \textit{all} elements of $X$.  In our notation $R(X;Y)=R_i(X)$ is the $i$th Rips complex of $Y$.)

\begin{definition}[coarse homology]
 The \textit{coarse homology groups} of a proper metric space $(X,d)$ are defined as
\[
 HX_*(X)=\varinjlim_i H\ind{*}{lf}(\vert \mathcal{U}_i\vert),
\]
\end{definition}

One can show that up to natural isomorphism $HX_*(X)$ does not depend on the choice of the coarsening sequence.  Moreover, quasi-isometries induce isomorphisms on coarse homology.

For coarse homology there exists an analogue of the Mayer-Vietoris sequence, but we must require the decomposition of a metric space ${X=C\cup D}$ to be \textit{coarsely excisive}, that is for every $r>0$ there exists an $R>0$ such that 
\[
 B_r(C)\cap B_r(D)\subset B_R(C\cap D).
\]

\begin{prop}[Lemma 3.9, \cite{Mitchener: Coarse homology theories}]\label{prop: coarse Mayer-Vietoris Mitchener}
 Let $X$ be a proper metric space and $X=C\cup D$ be a coarsely excisive decomposition.  Then there exists a coarse Mayer-Vietoris sequence
\[
 \ldots\rightarrow HX_n(C\cap D)\rightarrow  HX_n(C)\oplus HX_n(D)\rightarrow HX_n(X)\rightarrow HX\sub{n-1}(C\cap D)\rightarrow \ldots
\]
\end{prop}

It is often quite hard and very inconvenient to compute the coarse homology via an explicit coarsening sequence.  For certain spaces, though, Higson and Roe \cite{Higson-Roe: On the coarse Baum-Connes conjecture} showed that the coarse homology already equals the locally finite homology.  

\begin{definition}[uniform contractibility]
A metric space $X$ is called \textit{uniformly contractible} if for every $r>0$ there exists an $R\geq r$ such that $B_r(x)$ is contractible in $B_R(x)$ for every $x\in X$.
\end{definition}

\begin{definition}[bounded coarse geometry]\label{def: bounded coarse geometry metric space}
 A proper metric space has \textit{bounded coarse geometry} if there exists some $\ee>0$ such that the $\ee$-capacity of any ball of radius $r$ (i.e. the maximal number of disjoint $\ee$-balls in $B_r$) is bounded by some $c_r$.
\end{definition}

Following the terminology of \cite{Higson-Roe: On the coarse Baum-Connes conjecture}, we understand a \textit{bounded geometry complex} to be a metric simplicial complex, i.e.\ a simplicial complex equipped with the path metric induced from the canonical metric on the simplices, which has bounded coarse geometry.

\begin{prop}[Proposition 3.8, \cite{Higson-Roe: On the coarse Baum-Connes conjecture}]\label{prop: coarsening map is isomorphism}
 If $(X,d)$ is a uniformly contractible bounded geometry complex, then $HX_*(X,d)$ and $H\ind{*}{lf}(X)$ are isomorphic.
\end{prop} 

\begin{rmk}\label{rmk: locally finite homology of [0,infty)}
In the following sections we will use the fact that \linebreak ${H\ind{*}{lf}([0,\infty))= \{0\}}$. This is most easily seen by introducing a locally finite $\Delta$-homology analogous to the singular $\Delta$-homology presented in \cite{Hatcher: Algebraic Topology} and showing that it is isomorphic to the ordinary locally finite homology. The details can be found in \cite{Schmidt: Coarse topology of leaves of foliations}, Chapter $3$. 
\end{rmk}

\subsection{Coarse homology of Schweitzer's non-leaves}

In this section, we show that the coarse homology of the non-leaves constructed by Schweitzer in \cite{Schweitzer: Riemannian manifolds not quasi isometric to leaves in codimension one foliations} is non-finitely generated. Very similar arguments show that also the coarse homology of the non-leaves constructed by Zeghib in \cite{Zeghib: An expample of a 2-dimensional no leaf} is non-finitely generated.

Schweitzer's construction presented in Section \ref{section: Schweitzers bounded homology property} is closely related  to the following \textit{balloon space} consisting of spheres with radii tending towards infinity attached to the real line, which is defined in \cite{Hanke-Kotschick-Roe-Schick: Coarse topology enlargeability and essentialness}
\[
 B=[0,\infty)\bigcup\sub{i\in \N}\left( \cup_i S^n(i)\right),
\]
where $S^n(i)$ denotes the $n$-dimensional sphere of radius $i$ (see Figure \ref{fig: balloon space}).  They compute the coarse homology of $B$ to be
\[                                                                                                                 
  HX_n(B)\simeq \big(\prod_i \Z\big)/\big(\bigoplus_i \Z\big).
\]

\begin{figure}
\labellist
\pinlabel $1$ at 12 -3
\pinlabel $2$ at 44.5 -3
\pinlabel $3$ at 76.5 -3
\pinlabel $4$ at 108.5 -3
\pinlabel $5$ at 141 -3
\endlabellist
 \centering
 \includegraphics[width=13.5cm]{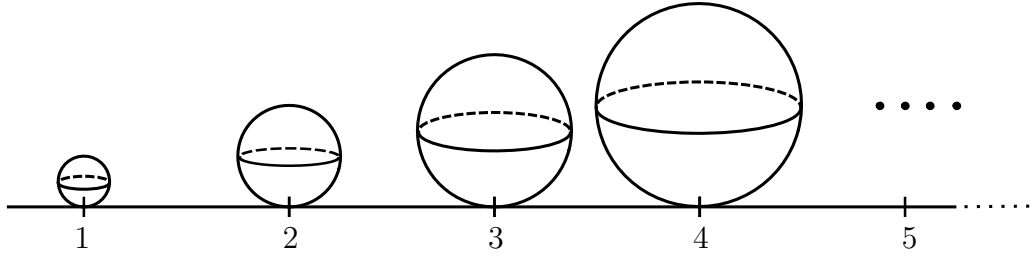}
\caption{The balloon space $B$.}
\label{fig: balloon space}
\end{figure}

The following proposition shows that an adapted balloon space $B'$ coarsely embeds into the non-leaves $(L,g_S)$ constructed by Schweitzer and shows that this embedding is injective in coarse homology.  

\begin{prop}\label{prop: g_SJ has infinitely generated coarse homology}
 Let $(L,g)$ be a complete connected open Riemannian manifold, of dimension $n\geq 2$ and let $g\sub{S}$ be the deformation of $g$ to the above described balloon metric performed along a ray in $L$.  Then $HX_n(L,g\sub{S})$ is not finitely generated.  In fact 
 \[
 HX_n(B)=\big(\prod_i\Z\big) / \big(\bigoplus_i\Z\big) \subset HX_n(L,g\sub{S}).
 \]
\end{prop}

\begin{proof} 
We shall use the notation of Section \ref{section: Schweitzers bounded homology property}. Let $\ga$ be a ray in $(L,g)$ and let $x_i=\ga(d_i), i=0, 1, \ldots$ be the sequence of points in $L$ where we have inserted balloons of radius $r_i$ tending towards infinity. Note that $(L,g\sub{S})$ and the space $L'$ which results from gluing the south poles of the $S^n(r_i)$ to the $x_i$ are quasi-isometric if we equip $L'$ with the path metric induced from $(L,g)$ and the $S^n(r_i)$.  Hence their coarse homology groups are isomorphic and we can work with $L'$ instead of $(L,g\sub{S})$.  Let $B'$ be the balloon space which is given by attaching an $n$-sphere $S^n(r_i)$ of radius $i$ to $d_i\in [0,\infty)$ for every $i\in \N$. In particular, we can consider $B'$ as a subspace of $L'$.  In what follows, $L$ is understood to be the `original' Riemannian manifold $(L,g)$.

Since $L'$ carries the path metric, $L'= L\cup B'$ is a coarsely excisive decomposition with $L\cap B'=[0,\infty)$ and by Proposition \ref{prop: coarse Mayer-Vietoris Mitchener} we have a Mayer-Vietoris sequence
\[
 \ldots \rightarrow HX_k([0,\infty))\rightarrow HX_k(L)\oplus HX_k(B')\rightarrow HX_k(L')\rightarrow HX\sub{k-1}([0,\infty))\rightarrow\ldots
\]
But $[0,\infty)$ is a uniformly contractible bounded geometry complex and thus Proposition \ref{prop: coarsening map is isomorphism} implies that $HX_k([0,\infty))\simeq H\ind{k}{lf}([0,\infty))$.  Since $H\ind{k}{lf}([0,\infty))=\{0\}$ for all $k\in \N$ (see Remark \ref{rmk: locally finite homology of [0,infty)}), the above sequence breaks down to 
\[
 0\lra HX_k(L)\oplus HX_k(B')\lra HX_k(L')\lra 0
\]
and thus 
\[
HX_k(L,g\sub{S})\simeq HX_k(L')\simeq  HX_k(L,g)\oplus HX_k(B').
\]
As in \cite{Hanke-Kotschick-Roe-Schick: Coarse topology enlargeability and essentialness}, we can choose a coarsening sequence $\calU_k$ for $B'$ such that $\vert \calU_k\vert$ is properly homotopic to $B'$ with the first $k$ spheres collapsed to their respective south soles.  Hence we find again that
\[
 HX_n(B')\simeq \varinjlim_k \big(\prod_i \Z\big)/\big(\bigoplus\ind{i=1}{k} \Z\big) \simeq \big(\prod_i \Z\big)/\big(\bigoplus_i \Z\big).
\qedhere
\]
\end{proof}

\subsection{Leaves with non-finitely generated coarse homology}

In this section we will prove Theorem \ref{thm: leaves with non-finitely generated HX1}. This will follow from the fact that there exist foliations of compact manifolds with leaves that have infinitely many ends.  We will show that the degree $1$ coarse homology of such leaves with any metric induced from the foliated manifold is non-finitely generated.  This follows from the following proposition according to which $k$ distinct ends in a proper geodesic space span a free Abelian subgroup of rank $k-1$ in the degree $1$ coarse homology.  (Cf.\ Prop.\ 2.25, \cite{Roe: Coarse Cohomology and Index Theory on Complete Riemannian Manifolds} for the analogous statement for coarse cohomology.)

\begin{prop}\label{prop: ends create degree 1 coarse homology}
 Let $(X,d)$ be a proper, connected, geodesic space  with \linebreak ${k\in \N\cup \{\infty\}}$ ends.  Then $HX_1\left((X,d);\Z\right)$ contains a subgroup isomorphic to $\oplus\ind{i=1}{k-1}\Z$.
\end{prop}

Note that not all elements in the degree one coarse homology originate from the ends of the space.  Take for example $X$ to be the $1$-dimensional balloon space from \cite{Hanke-Kotschick-Roe-Schick: Coarse topology enlargeability and essentialness}, then $X$ has just one end, but the coarse homology ${HX_1(X)=\prod_j \Z/\oplus_j\Z}$ is not even finitely generated.

The proof of Proposition \ref{prop: ends create degree 1 coarse homology} will occupy the remainder of this section.  Before, we show how the proposition implies Theorem \ref{thm: leaves with non-finitely generated HX1}:

\begin{proof}[Proof of Theorem \ref{thm: leaves with non-finitely generated HX1}]
By a result of Cantwell and Conlon (Theorem A, \cite{Cantwell-Conlon: Endsets of leaves}) every compact, totally disconnected, metrizable space $E$ can be realized as the endspace of a $2$-dimensional leaf $\Sigma$ in a codimension one foliation of a compact $3$-manifold $M$.  In particular there exist surfaces with infinitely many ends which are diffeomorphic to a leaf in a foliation of a compact $3$-manifold.  (Such leaves can also be constructed more elementarily by turbulizing a linear foliation of $T^3$ by dense cylinders (Example 4.3.10, \cite{Candel-Conlon: Foliations I}).)  Simply by taking products $M\times N^d\times N^c$ with compact manifolds and considering the leaf $\Sigma \times N^d$, we find that in every leaf-dimension at least $2$ and every codimension at least $1$, there exist leaves in foliations of compact manifolds with leaves that have infinitely many ends.

Let $(M, \scrF)$ be such a foliation and $L$ be a leaf of $\scrF$ with infinitely many ends.  Then every metric $g$ on $M$ induces a complete metric $\iota^*g$ on $L$.  In particular $(L,\iota^*g)$ is a proper geodesic space with infinitely many ends.  Hence we can apply Proposition \ref{prop: ends create degree 1 coarse homology} to see that $HX_1(L,\iota^*g)$ contains a free Abelian subgroup of infinite rank and hence cannot be finitely generated.
\end{proof}

\subsubsection{Ends of a topological space}\label{section: ends of a topological space}

Let $X$ be a topological space.  By an \textit{end of $X$} we mean the equivalence class of a proper ray $r\colon [0,\infty)\lra X$, where two rays $r, r'$ are equivalent if for each compact subset $K\subset X$ there exist a $t_K>0$ such that $r([t_K,\infty))$ and $r'(t_K,\infty))$ lie in the same path component of $X\minus K$.  We denote by $\calE(X)$ the set of ends and let $e(X)$ be the cardinality of $\calE(X)$.

The number of ends is in general not a quasi-isometry invariant of metric spaces and thus is in general not detected by the coarse homology.  Let for example $X=\left(\Q\times \R\right)\cup \left( \R\times \{0\}\right)$ with the metric induced from $\R^2$.  Then $X$ is a connected metric space and there exists a bijection from $\mathcal{E}(X)$ to $\Q\cup\Q$. But $X$ is quasi-isometric to $\R^2$ which has only one end.  For proper geodesic spaces this does not happen.

\begin{lemma}[8.29 Proposition, \cite{Bridson-Haefliger: Metric Spaces of Non-Positive Curvature}]\label{lemma: number of ends quasi-isometry invariant}
 For proper geodesic spaces $X$ and $Y$, every quasi-isometry $f\colon X\lra Y$ induces a bijection $\calE(X)\lra \calE(Y)$.
\end{lemma}

It turns out that if $f$ is just the isometric embedding of a subspace, the induced map on the end spaces can be chosen as the embedding of rays.  It is moreover not hard to see that a metric space is quasi-isometric to any of its coarsenings. Thus we get the following remark:

\begin{rmk}\label{remark: remark on ends of coarsenings}
 Let $\calU_1, \calU_2,\ldots$ be a coarsening sequence for $X$, then the above lemma implies that $e(X)= e(\vert \calU_i\vert)$.  That is, by coarsening a geodesic space we do not create or loose ends.  Moreover, for $\vert U_i\vert =R_i(X)$ the bijections $\calE(R_i(X))\lra \calE(R\sub{i+1}(X))$ are induced by the inclusion of rays from $R_i(X)$ into $R\sub{i+1}(X)$.
\end{rmk}

\subsubsection{Proof of Proposition \ref{prop: ends create degree 1 coarse homology}}

We will prove Proposition \ref{prop: ends create degree 1 coarse homology} as follows.  First, we will show how two distinct ends of a space give rise to a locally finite $1$-chain and when we are dealing with a locally finite simplicial complex $S$, these chains will generate a free Abelian subgroup on $e(S)-1$ generators in the degree $1$ locally finite homology.  If we let $X$ be a proper geodesic space, Remark \ref{remark: remark on ends of coarsenings} implies that $e(X)=e(R_i(X))$ and hence all locally finite homology groups $H\ind{1}{lf}(R_i(X))$ contain a free rank $e(X)-1$ Abelian subgroup.  It will be easy to see, that these subgroups map to each other under the coarsening maps and hence $HX_1(X)$ also contains a free rank $e(X)-1$ Abelian subgroup.  In particular, for infinitely many ends, $HX_1(X)$ cannot be finitely generated.
 
\begin{lemma}\label{lemma: rank k-1 abelian subgroup}
Let $S$ be a locally finite simplicial complex with ${k\in \N\cup \{\infty\}}$ ends.  Then    $H\ind{1}{lf}(S;\Z)$ contains a subgroup isomorphic to $\oplus\ind{j=1}{k-1}\Z$.
\end{lemma}

\begin{proof}
Let $r_1,r_2,\ldots$ be the ends of $S$.  By subdividing $r_j$ as $\sum_n r_j|\sub{[n,n+1]}$ each end can be viewed as a $1$-chain, which we again denote by $r_j$.  These chains are locally finite because the maps $r_j\colon [0,\infty)\lra $ are proper. Without loss of generality, we may assume that all rays $r_j$ emanate from the same point $x\in S$. Then $z\sub{j,j+1}:= -r_j + r\sub{j+1}$ are locally finite $1$-cycles for all $j=1,\ldots, k-1$. We claim that the $z_{j,j+1}$ generate a free Abelian subgroup of rank $k-1$ in $H\ind{1}{lf}(S)$.

Let $K\subset S$ be chosen such that $r_1,\ldots,r_k\in \calE(S)$ eventually map to distinct components of $S\setminus K$ and let $K'\supset K$ be a compact neighbourhood of $K$ that properly deformation retracts onto $K$ (this is possible since $S$ is a locally finite simplicial complex) and consider the locally finite Mayer-Vietoris sequence for $S= K'\cup (S\setminus K)$:
\[
 \ldots \lra H\ind{1}{lf}(S)\stackrel{\partial}{\lra} H\ind{0}{lf}(\overline{S\minus K}\cap K')\stackrel{\ff}{\lra} H\ind{0}{lf}(\overline{S\minus K})\oplus H\ind{0}{lf}(K')\lra \ldots 
\]
Then $\overline{S\minus K}\cap K'$ has $k$ compact components corresponding to the $r_1,\ldots, r_k$ and hence $H\ind{0}{lf}(\overline{S\minus K}\cap K')$ contains a free Abelian subgroup of rank $k$, which is, without loss of generality, generated by points of the from $r_j(n_j), {j=1, \ldots, k}$ and some $n_j\in \N$.  The first component of $\ff$ is the zero map since $r_j(n_j)=\partial \big(\sum\sub{n\geq n_i} r_j|\sub{[n,n+1]}\big)$ the second component of $\ff$ is the map $(m_1,\ldots,m_k)\mapsto m_1+\ldots + m_k$.  Hence $\im(\partial)=\ker(\ff)\supset \oplus\ind{j=1}{k-1}\Z\{r\sub{j+1}(n\sub{j+1})-r_j(n_j)\}$.

Recall that the boundary map of the Mayer-Vietoris sequence is defined by subdividing $1$-chains on $S$ into a sum of $1$-chains on $K'$ and on $S\minus K$ and mapping to the boundary of either summand.  Since $z\sub{j,j+1}$ can be decomposed as 
\[
\Big( -\sum\sub{n< n_j} r_j|\sub{[n,n+1]} + \sum\sub{n< n\sub{j+1}} r\sub{j+1}|\sub{[n,n+1]}\Big) 
+ \Big(-\sum\sub{n\geq n_j}r_j|\sub{[n,n+1]} + \sum\sub{n\geq n\sub{j+1}}r\sub{j+1}|\sub{[n,n+1]}\Big),
\]
where the first summand lies in $K'$ and the second in $S\setminus K$, we find that
\[
r\sub{j+1}(n\sub{j+1})-r_j(n_j)=\partial z\sub{j,j+1}.
\]
In particular, the $z_{j,j+1}$ map to the generators of $\oplus\ind{j=1}{k-1}\Z\{r\sub{j+1}(n\sub{j+1})-r_j(n_j)\}$ and thus generate a free Abelian subgroup of rank $k-1$ in $H\ind{1}{lf}(S)$.  Thus we have an increasing, and if $e(S)<\infty$ eventually stationary, sequence
\[
\Z\{ z_{1,2}\} \;\subset\; \Z\{z_{1,2}\}\dirsum \Z\{z_{2,3}\} \;\subset \ldots \subset\; \bigoplus\ind{j=1}{l-1}\Z\{z_{j,j+1}\}\;\subset\ldots\subset H\ind{1}{lf}(S;\Z).
\]
Hence $\oplus\ind{j=1}{e(S)-1}\Z\{z_{j,j+1}\}\subset H\ind{1}{lf}(S;\Z)$.
\end{proof}

The proof of Proposition \ref{prop: ends create degree 1 coarse homology} is now an easy consequence of the above facts.

\begin{proof}[Proof of Proposition \ref{prop: ends create degree 1 coarse homology}]
Recall that for a direct system of groups  \linebreak ${G_1\lra G_2\lra \ldots}$, elements $\al_1,\ldots, \al_n\in \varinjlim G_i$ generate a rank $n$ Abelian subgroup in $\varinjlim G_i$ if for all sufficiently large $i$, the representatives $a_j\in G_i$ of the $\al_j$ generate a rank $n$ Abelian subgroup in $G_i$. 

For simplicity and geometric clearness, we take $\{R_i(X)\}_i$ as coarsenings of $X$ and compute $HX_1(X)$ via $\varinjlim H\ind{1}{lf}(R_i(X))$. By Remark \ref{remark: remark on ends of coarsenings} the $R_i(X)$ all have $k=e(X)$ ends and Lemma \ref{lemma: rank k-1 abelian subgroup} then shows that the degree one locally finite homology groups $H\ind{1}{lf}(R_i(X)), i\geq 1$, each contain a rank $k-1$ free Abelian subgroup.  Moreover, we have seen that the coarsening maps $ R_i(X)\hookrightarrow R\sub{i+1}(X)$ map the ends of $R_i(X)$ to the ends of $R\sub{i+1}(X)$ and hence the generators of the rank $k-1$ free Abelian subgroup in $H\ind{1}{lf}(R_i(X))$ constructed in Lemma \ref{lemma: rank k-1 abelian subgroup} to the generators of the respective subgroup in $H\ind{1}{lf}(R\sub{i+1}(X))$.  Thus the equivalence classes $[z\sub{j,j+1}]\in \varinjlim H\ind{1}{lf}(R_i(X))=HX_1(X)$ generate a free Abelian subgroup of rank $k-1$.
\end{proof}

\section{Non-leaves with trivial coarse homology}\label{section: non-leaves with trivial coarse homology}

In this section we prove Theorem \ref{thm: non-leaves with unchanged coarse homology}.  By the results of Schweitzer presented in Section \ref{section: Schweitzers bounded homology property} of this article, every manifold that is diffeomorphically quasi-isometric to a leaf of a codimension one $C^{2,0}$ foliation must satisfy the bounded homology property.  Thus Theorem \ref{thm: non-leaves with unchanged coarse homology} follows directly from the following lemma.

\begin{lemma}\label{lemma: deformation to non-bounded homology without changing coarse homology}
 On every non-compact manifold $(M,g)$ of bounded geometry, $\dim M\geq 3$, there exists a metric $g'$ not satisfying the bounded homology property such that $HX_*(M,g')\simeq HX_*(M,g)$.  Moreover, we can choose $g'$ such that it has the same growth type as $g$.
\end{lemma}

The proof of Lemma \ref{lemma: deformation to non-bounded homology without changing coarse homology} is carried out in three steps. We first show how to construct the metric $g'$, then show that it doesn't satisfy the bounded homology property and finally prove that its coarse homology is isomorphic to that of the original metric.

\textbf{Construction of the metric $g'$:}  The construction of $g'$ is analogous to that of the non-leaf metric constructed by Schweitzer in \cite{Schweitzer: Riemannian manifolds not quasi isometric to leaves in codimension one foliations}, Section 4. Instead of balloons we use a kind of tree-manifolds $\frakT_k$ constructed below.

Recall that a \textit{rooted tree} is a tree with a distinguished vertex, which we call the \textit{root}.  The \textit{leaves} of a connected tree are the vertices of degree $1$, while we do not want to consider the root as a leaf.  The \textit{height} of a leaf is its distance from the root, where we let each edge have length equal to $1$.  By the \textit{perfect binary tree of height $k$} we mean the rooted tree, where the root and the leaves have degree $1$, while every other vertex has exactly $2$ children (i.e.\ has degree $3$) and every leaf has height $k$ (see Figure \ref{fig: perfect binary trees}).  (Note that commonly the root in binary trees is also required to have $2$ children.)  Denote the perfect binary tree of height $k$ by $\calT_k$.

\begin{figure}
 \centering
 \includegraphics[width=12cm]{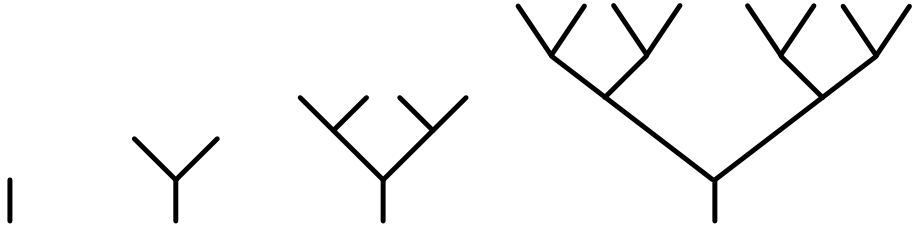}
\caption{$\calT_1, \ldots, \calT_4$ (edge lengths are not true to scale).}
\label{fig: perfect binary trees}
\end{figure}

By gluing cylinders $S^{n-1}\times [-1,1]$ to the edges and connecting them through spheres with three punctures at the vertices and finally adding disks to the leaves, we thicken the trees $\calT_k$ to get Riemannian manifolds of bounded geometry, which we denote by $\frakT_k$ (see Figure \ref{fig: the building block T_3 with dooted T_3}).  For technical reasons, we rescale the edge emanating from the root to have length $2k$ and glue a cylinder $S^{n-1}\times [-k,k]$ to it. Then $\frakT_k$ is quasi-isometric to $\calT_k$ with a rescaled edge, but topologically $\frakT_k$ is just a ball.

The tree-manifolds $\frakT_k$ are glued to $M$ as in Section \ref{section: Schweitzers bounded homology property}. We use the same notation. Let $\ga$ be a ray in $(M,g)$ and set $x_k=\ga(d_k), i=0,1,\ldots$ such that ${d(x\sub{k-1},x_k)>3d}$, where $\nicefrac{1}{2}>d>0$ is a lower bound for the injectivity radius of $M$. Then the distance balls $B_d(x_k)$ are also topological balls.  We replace the original metric $g$ on $B_d(x_k)$ by the metric induced from $\frakT_k$ (after possibly interpolating the metric on $S^{n-1}\times [-k, 0]$ between the diameter of $\partial B_d(x_k)$ and the diameter of the sphere factor of $S^{n-1}\times [-k,k]$).   By the same arguments as in \cite{Schweitzer: Riemannian manifolds not quasi isometric to leaves in codimension one foliations}  this can be performed in such a way, that the new resulting metric $g'$ on $M$ is smooth and again of bounded geometry. By placing the $\frakT_k$ sufficiently far apart, we can make sure that $g$ and $g'$ have the same growth type.

\begin{figure}
\labellist
\pinlabel $D^n$ at -12 300
\pinlabel $S^{n-1}\cart [-1,1]$ at 60 175
\pinlabel $S^{n-1}\times [-k,k]$ at 105 60
\endlabellist
\centering
    \includegraphics[width=10cm]{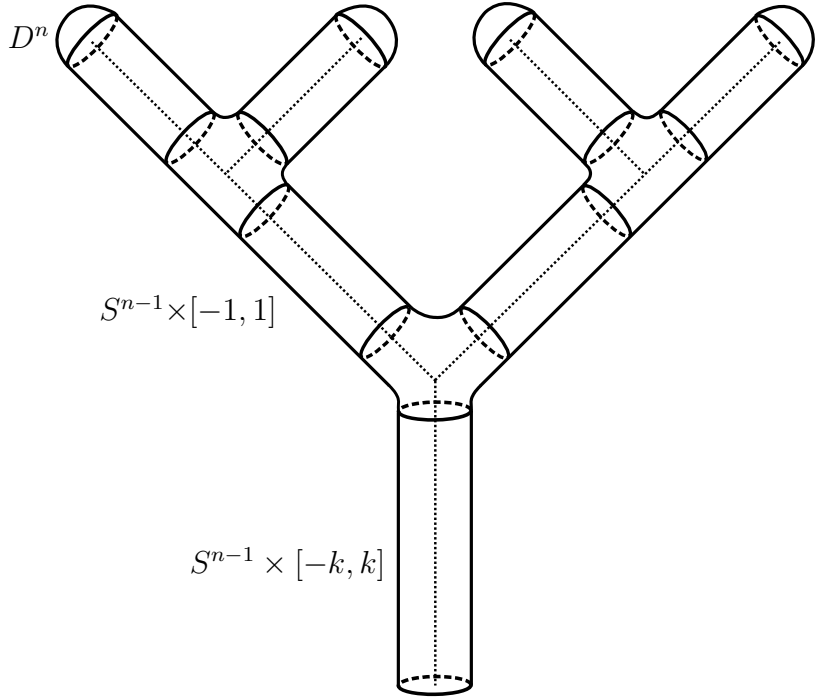}
  \caption{The manifold $\mathfrak{T}_3$ with dotted tree $\mathcal{T}_3$.}
\label{fig: the building block T_3 with dooted T_3}
\end{figure} 

\noindent\textbf{Proof that $(M,g')$ does not satisfy the bounded homology property:} In what follows, denote by $\mathfrak{T}_k'$ the manifold $\mathfrak{T}_k$ without the lower part ${S^{n-1}\times [-k,0)}$ of the cylinder starting from the root of $\mathcal{T}_k$.  Thus $\partial \mathfrak{T}_k'=S^{n-1}\times \{0\}$.  It suffices to show that the Morse-$\bb$-volume of the $\mathfrak{T}_k'$ is unbounded for every $\bb>0$ as $k$ goes to infinity.  For then, the $\mathfrak{T}_k'$ form a sequence of closed codimension $0$ submanifolds with $\betavol(\partial\mathfrak{T}_k')=\betavol(S^{n-1})$, while there exists no constant $L>0$ such that $\mvolbeta(\mathfrak{T}_k')\leq L$.  It is easy to verify that the $\mathfrak{T}_k'$ satisfy the conditions i)-iii) in Definition \ref{def: bounded homology property}:  We have already seen that $\betavol(\partial \mathfrak{T}_k')$ is constant, and as $n\geq 3$ both  $\mathfrak{T}_k'\approx D^n$ and $\partial\mathfrak{T}_k'\approx S^{n-1}$ are connected and simply connected,
 thus conditions i) and ii) are fulfilled.  To see that the boundary of $\mathfrak{T}_k'$ has a large tubular neighbourhood, i.e. satisfies condition iii), we use that the length of the cylinder at the root of $\mathcal{T}_k$ has length $2k$.  Thus for every $k> \bb$, this cylinder $S^{n-1}\times [-k,k]$ is a tubular neighbourhood of $\partial \mathfrak{T}_k'$ that contains the $\bb$-neighbourhood $B\sub{\bb}(\partial \mathfrak{T}_k)$.

To prove that $\mvolbeta(\mathfrak{T}_k')$ goes to infinity, we show that for every continuous (and in particular, for every Morse function) $f\colon \mathfrak{T}_k'\lra \R$ there exists an $x\in \R$ such that $f$ takes the value $x$ on at least $\lceil \frac{k}{2}\rceil$ cylinders ${S^{n-1}\times [-1,1]},{S^{n-1}\times [0,k]}$, disks or punctured spheres in $\frakT_k$.  We will henceforth refer to these four types of pieces as \textit{building blocks} of $\frakT_k$. This statement is sufficient to prove that $\mvolbeta(\mathfrak{T}_k')$ goes to infinity because every ball of radius $\bb$ in $\frakT_k$ contains at most a bounded number of building blocks, say $c\sub{\bb}$, and for $x$ as above
\[
\betavol(f^{-1}(x))\geq \frac{\lceil \frac{k}{2}\rceil}{c\sub{\bb}}
\]
holds and hence 
\[
\mvolbeta(\mathfrak{T}_k')\geq \frac{\lceil \frac{k}{2}\rceil}{c\sub{\bb}}.
\]
Note that we do neither require our functions to vanish on $\partial \mathfrak{T}_k'$.  This will enable us to do induction to the trees of lower height lying inside of $\mathcal{T}_k$.  We define $u(k)$ to be the smallest natural number such that for any continuous function $f\colon \frakT_k\lra \R$ there exists some $x_f\in \R$ such that $f$ takes the value $x_f$ on at least $u(k)$ building blocks and make the following claim:   

\begin{lemma}
 $u(k+2)>u(k)$.
\end{lemma}

\begin{proof}
\begin{figure}
\labellist
\pinlabel $\mathfrak{T}\ind{k}{(1)}$ at 13 213
\pinlabel $\mathfrak{T}\ind{k}{(2)}$ at 113 213
\pinlabel $\mathfrak{T}\ind{k}{(3)}$ at 205 213
\pinlabel $\mathfrak{T}\ind{k}{(4)}$ at 305 213

\endlabellist
\centering
\includegraphics[width=12cm]{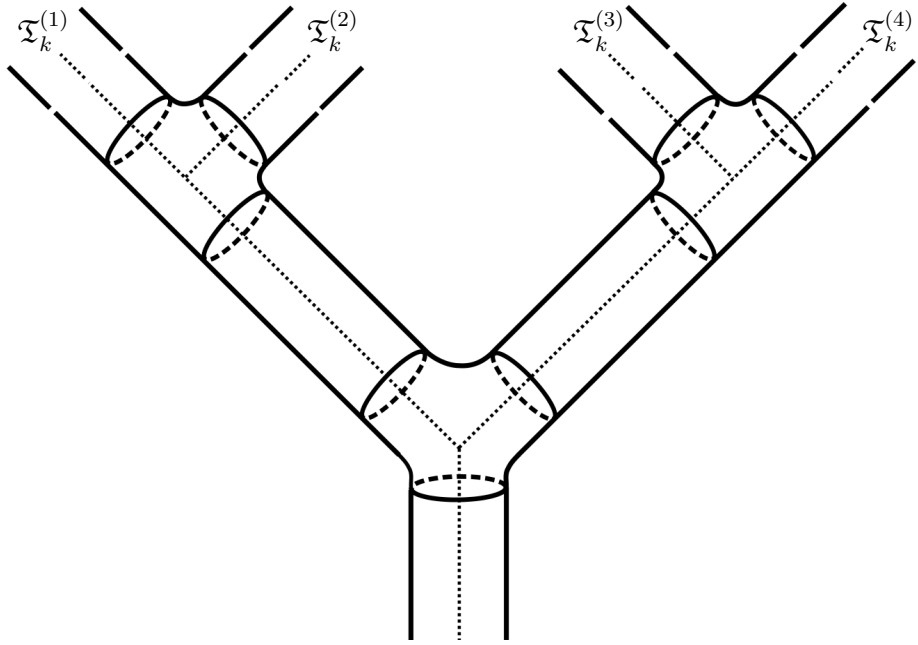}
\caption{$\mathfrak{T}\sub{k+2}$ with four copies of $\mathfrak{T}\sub{k}$.}
\label{fig: k+2-baum mit 4 mal k-baum}
\end{figure}

Let $f\colon \mathfrak{T}\sub{k+2}'\lra \R$.  Note that $\mathfrak{T}\sub{k+2}'$ contains four copies $\mathfrak{T}\ind{k}{'(1)}, \ldots,\mathfrak{T}\ind{k}{'(4)}$ of $\mathfrak{T}_k'$ (up to scaling the length of the edge leading to the root) as shown in Figure \ref{fig: k+2-baum mit 4 mal k-baum}.  Let $t\sub{\max}\in \mathfrak{T'}\sub{k+2}$ be such that $f(t_{\max})=\max f$ and $t\sub{\min}$ analogously.  Then there exists a $\mathfrak{T}\ind{k}{'(s)}$, say $\mathfrak{T}\ind{k}{'(1)}$, such that the geodesic between $x\sub{\min}$ and $x\sub{\max}$ does not intersect $\mathfrak{T}\ind{k}{'(1)}$.  But for $f\mid\sub{\mathfrak{T}\ind{k}{'(1)}}\colon \mathfrak{T}\ind{k}{'(1)}\lra \R$ there exists some $x\in \R$ and $u(k)$ building blocks in $\mathfrak{T}\ind{k}{'(1)}$ such that $f\mid\sub{\mathfrak{T}\ind{k}{'(1)}}$ takes the value $x$ on these blocks.  But by construction $f(t\sub{\min})\leq x \leq f(t\sub{\max})$ and since the geodesic between $t\sub{\min}$ and $t\sub{\max}$ does 
not intersect $\mathfrak{T}\ind{k}{'(1)}$, there exists an additional edge in $\mathfrak{T}\sub{k+2}'\setminus \mathfrak{T}\ind{k}{'(1)}$ such that $f$ takes the value $x$ on this edge.  Hence $u(k+2)\geq u(k)+1$.
\end{proof}

\noindent\textbf{Proof that $HX_*(M,g')\simeq HX_*(M,g)$:} Note that $(M,g')$ is coarsely quasi-isometric to $(M,g)$ with the trees $\mathcal{T}_k$ attached to the $x_k$.  Denote this metric space by $\mathfrak{M}$.  Then $HX_*(M,g')\simeq HX_*(\mathfrak{M})$.  But $\mathfrak{M}$ has a coarsely excisive decomposition into $(M,g)$ and the ray $\ga$ with the trees $\calT_k$ attached to ${x_k:=\ga(d_k)}$, which in turn is isometric to 
\[
 \mathcal{T}':=[0,\infty)\bigcup\sub{d_k} \left( \cup_k \mathcal{T}_k\right),
\]
where $\calT'$ carries the path metric induced from the $\calT_k$ and $[0,\infty)$. Since ${(M,g)\cap \mathcal{T}'}$ is isometric to $[0,\infty)$, Proposition \ref{prop: coarsening map is isomorphism} yields a coarse Mayer-Vietoris sequence
\begin{align*}
 \ldots &\lra  HX_k([0,\infty))\lra HX_k(M,g)\oplus HX_k(\mathcal{T}')\lra HX_k(\mathfrak{M})\lra \\ 
 &\hspace{10mm}\lra HX_{k-1}([0,\infty))\lra \ldots
\end{align*}
Since $HX_k([0,\infty))=\{0\}$ (see the proof of Proposition \ref{prop: g_SJ has infinitely generated coarse homology}), we have isomorphisms $HX_k(\mathfrak{M})\simeq HX_k(M,g)\oplus HX_k(\mathcal{T}')$. It thus remains to show that $HX_k(\mathcal{T}')$ vanishes.

$\mathcal{T}'$ naturally has the structure of a metric simplicial complex and this space has bounded coarse geometry in the sense of Definition \ref{def: bounded coarse geometry metric space}: Any ball of radius $r$ in $\mathcal{T}'$ intersects at most $2^{\lceil r\rceil+2}$-many edges.  Since every subset of diameter $1$ intersects at least one edge and since every edge is intersected by at most two disjoint sets of diameter $1$, the $1$-capacity of every ball of radius $r$ in $\mathcal{T}'$ is bounded for any given $r$.  Hence $\mathcal{T}'$ is a bounded geometry complex.  Moreover every ball in $\mathcal{T}'$ is contractible within itself, consequently $\mathcal{T}'$ is in particular uniformly contractible.   Thus we can apply Proposition \ref{prop: coarsening map is isomorphism} to find that $HX_*(\mathcal{T}')\simeq H\ind{*}{lf}(\mathcal{T}')$.  Though collapsing each $\mathcal{T}_k$ to its respective root is \textit{not} a quasi-isometry, it is a proper homotopy equivalence and thus $H\ind{*}{lf}(\mathcal{T}')\simeq H\ind{*}{lf}([0,\infty))=\{0\}$.

Summarizing, we have the following sequence of isomorphisms:
\[
 HX_k(M,g')\simeq HX_k(\mathfrak{M}) \simeq HX_k(M,g)\oplus HX_k(\calT') \simeq HX_k(M,g).
\]
This finishes the proof of Theorem \ref{thm: non-leaves with unchanged coarse homology}. \qed

\end{document}